\documentclass[reqno,a4paper]{amsart}
\usepackage[english,activeacute]{babel}
\usepackage{amssymb,amsmath,amsthm,amsfonts,mathrsfs,scalefnt,graphicx,graphics,color,hyperref,}
\usepackage{lmodern}

\usepackage{tikz}
\usetikzlibrary{trees}
\usetikzlibrary{shapes}
\usepackage{pst-tree}
\usepackage{pst-all}
\usepackage{tabularx}
\usepackage[T1]{fontenc}
\usepackage{url}  
\usepackage[latin1]{inputenc}
\theoremstyle{plain}
\newtheorem{theorem}{Theorem}[section]
\newtheorem{lemma}[theorem]{Lemma}
\newtheorem{proposition}[theorem]{Proposition}
\newtheorem{corollary}[theorem]{Corollary}

\newtheorem{remark}[theorem]{Remark}
\numberwithin{equation}{section}

\newcommand{\Nb}  {{\mathbb N}}
\newcommand{\Rb}  {{\mathbb R}}

\newcommand{\Xb}  {{\mathbb X}}
\newcommand{\As} {{\mathcal A}}

\newcommand{\Fs} {{\mathcal F}}

\newcommand{\Ps} {{\mathcal P}}

\newcommand{\Ts} {{\mathcal T}}

\newcommand{\Ys} {{\mathcal Y}}

\DeclareMathOperator{\Var}{Var}

\newcommand{\ep}{\varepsilon}
\renewcommand{\phi}{\varphi}

\newcommand{\ind}{1\!\kern-1pt \mathrm{I}}
\newcommand{\rsto}{]\!\kern-1.8pt ]}
\newcommand{\lsto}{[\!\kern-1.7pt [}

\newcommand\F{\mbox{I\kern-2pt F}}

\title[Asymptotic proportion of arbitrage points]{Asymptotic proportion of arbitrage points in fractional binary markets}
\author{Fernando Cordero}
\address{Faculty of Technology, University of Bielefeld, Universit\"{a}tsstr. 25, 33615 Bielefeld, Germany}
\email{fcordero@techfak.uni-bielefeld.de}

\author{Irene Klein}
\address{Department of Statistics and Operations Research, University of Vienna, Oskar-Morgensternplatz 1, A-1090
Vienna, Austria}
\email{irene.klein@univie.ac.at}

\author{Lavinia Perez-Ostafe}
\address{Department of Mathematics, ETH Zurich, R\"{a}mistrasse 101, 8092 Zurich, Switzerland}
\email{lavinia.perez@math.ethz.ch}
\thanks{The third author gratefully acknowledges financial support from the Austrian Science Fund (FWF) under grant J3453-N25}
\date{\today}%
\begin{document}
\begin{abstract}
A fractional binary market is a binary model approximation for the fractional Black-Scholes model, which Sottinen constructed with the help of a Donsker-type theorem. In a binary market the non-arbitrage condition is expressed as a family of conditions on the nodes of a binary tree. We call ``arbitrage points'' the nodes which do not satisfy such a condition and ``arbitrage paths'' the paths which cross at least one arbitrage point. In this work, we provide an in-depth analysis of the asymptotic proportion of arbitrage points and arbitrage paths. Our results are obtained by studying an appropriate rescaled disturbed random walk.
\end{abstract}
\maketitle

\section{Introduction}
In the classical theory of mathematical finance a crucial role is played by the notion of \textit{arbitrage}, which is the cornerstone of the option pricing theory that goes back to F.~Black, R.~Merton and M.~Scholes \cite{Bl:Sch}. %Significant work has been done to develop this theory. 
In the case of binary markets, the absence of arbitrage is completely characterized by Dzhaparidze in \cite{Dzh}. Intuitively, a binary market is a market in which the stock price process $(S_n)_{n=0}^N$ is an adapted stochastic process with strictly positive values and such that at time $n$ the stock price evolves from $S_{n-1}$ to either $\alpha_n\, S_{n-1}$ or $\beta_n\, S_{n-1} $, where $\beta_n <\alpha_n$.

One advantage of working with binary markets is given, on one hand, by their simplicity and, on the other hand, by their flexibility to approximate more complicated models. In particular, this is possible for Black-Scholes type markets that are driven by a process, for which we dispose of a random walk approximation. Examples of this are the fractional Brownian motion and the Rosenblatt process, as one can see in \cite{Sotti} and \cite{Totu} respectively.

In this paper we provide an in-depth analysis of fractional binary markets, which are defined by Sottinen \cite{Sotti} as a sequence of binary models approximating the fractional Black-Scholes model, i.e.~a Black-Scholes type model where the randomness of the risky asset comes from a fractional Brownian motion. Along this work we assume that the Hurst parameter $H$ is strictly bigger than $1/2$. In this case, the fractional Brownian motion exhibits self-similarity and long-range dependence, properties that were observed in some empirical studies of financial time series (see \cite{Con} and \cite{WiTaTe}). 
%For this reason these models are thought to describe real world markets in a better way, and hence their use substantially increased. However, 
Since the fractional Brownian motion fails to be a semimartingale, the fractional Black-Scholes model admits arbitrage opportunities, a drawback that can be corrected if, e.g. one introduces transaction costs.

In \cite{Sotti} Sottinen constructs the fractional binary markets by giving an analogue of the Donsker theorem, where the fractional Brownian motion is approximated in distribution by a ``disturbed'' random walk. Sottinen proves that the arbitrage opportunities do not only appear in the limiting model, but also in the sequence of fractional binary markets. 

According to \cite{Dzh}, in a binary market, the absence of arbitrage can be written as a family of conditions on the nodes of a binary tree. We call an ``arbitrage point'' a node in the binary tree which does not satisfy the corresponding non-arbitrage condition. An ``arbitrage path'' is a path that crosses at least one arbitrage point. By \cite{Sotti} we know that, for each fractional binary market in the sequence, the associated set of arbitrage points is not empty.

The study of the set of arbitrage points provides a way to quantify arbitrage, a research direction which goes a step further than the classical question of existence of arbitrage. 

The aim of this paper is to study qualitative and quantitative properties of the sets of arbitrage points and paths for the fractional binary market. First, we prove that starting from any point in the binary tree we reach an arbitrage point by going enough times only up or only down (Proposition \ref{arbpt}). This generalizes the result of Sottinen, who showed the existence of arbitrage starting only from the root of the tree. This gives information about the structure of the set of arbitrage points and implies that its cardinality is asymptotically infinite. Next, we study the limit behaviour of the proportion of arbitrage points. The latter is expressed in terms of a rescaled random walk, which we show converges in law. The characterization of the asymptotic proportion of arbitrage points then follows (Theorem \ref{proparbpt}). We also take a closer look to the previous limit when $H$ tends to $1/2$ and when $H$ tends to $1$ (Proposition \ref{proparbptH}). Finally, making use of the $0-1$ Kolmogorov law, we show that when $H$ is close to $1$, a.s. a path in the binary tree crosses an infinite number of arbitrage points, and when $H$ is close to $1/2$, a.s. a path in the binary tree crosses an infinite number of non-arbitrage points (Theorem \ref{asp}).

We believe that our asymptotic results open a way to a better understanding of the arbitrage behaviour in the limiting market. Since the proportion of arbitrage points remains strictly positive in the limit, one could expect that the sequence of sets of arbitrage points converges in a proper way to a set encoding the arbitrage structure of the fractional Black-Scholes model.

Another possible direction, in which our results may turn useful is the study of arbitrage in the fractional binary markets under transaction costs.
As mentioned by Sottinen, one may expect that the arbitrage disappears when transaction costs are taken into account. This latter problem was treated in its most generality in \cite{CKO}, where a characterization of the smallest transaction cost (called ``critical'' transaction costs) starting from which the arbitrage is eliminated is provided. However, since the parameters of the model depend on time and space, this characterization does not give a closed-form solution, but reduces to solving an optimization problem in a binary tree. The complexity of this problem increases with the number of arbitrage points, and, hence, the understanding of qualitative and quantitative properties gives us an insight to this more complicated problem. 

The paper is organized as follows. In Section \ref{Sec2}, we recall the fractional binary markets as defined in \cite{Sotti}. In Section \ref{Sec3}, we introduce the notion of arbitrage point and arbitrage path and we state our main results: Proposition \ref{arbpt}, Theorem \ref{proparbpt}, Proposition \ref{proparbptH} and Theorem \ref{asp}. The remaining of the paper is devoted to their proofs. In Section \ref{Sec4}, we prove that the parameters of the fractional binary markets satisfy a scaling property. This helps us to get rid of the dependence on the size of the fractional binary market. We finish this section with the proof of Proposition \ref{arbpt}. In Section \ref{Sec5}, we relate the proportion of arbitrage points with a rescaled random walk and study its limit behaviour. Finally, we prove Theorem \ref{proparbpt} and Proposition \ref{proparbptH}. Section \ref{Sec6} contains the proof of our last main result, Theorem \ref{asp}, concerning the asymptotic proportion of arbitrage paths. We end the paper with an Appendix enclosing the most technical auxiliary results needed along this work.

%%%section2%%%%%%%%%%%%%%%%
\section{Fractional binary markets}\label{Sec2}
Sottinen introduces in \cite{Sotti} the fractional binary markets as a sequence of binary markets approximating the fractional Black-Scholes model. The latter is a Black-Scholes type model where the randomness of the risky asset is determined by a fractional Brownian motion. More precisely, the dynamics of the bond and stock are given by:
\begin{equation}\label{fbse}
 dB_t=r(t)\,B_t\, dt\quad\textrm{and}\quad dS_t^H=(a(t)dt+\sigma\,dZ^H_t)\, S_t^H,\quad t\in[0,1],
\end{equation}
where $\sigma>0$ is a constant representing the volatility and $Z^H$ is a fractional  Brownian motion of Hurst parameter $H\in(1/2,1)$. In this case, the increments of the fractional Brownian motion are positively correlated and exhibit long-range dependence, see e.g. \cite{MVN} and \cite{Novavi}. The functions $r$ and $a$ are deterministic and represent the interest rate and the drift of the stock, respectively. We assume moreover that the interest rate is constant equal to 0 and that the drift $a$ is continuous.

For each $N>1$, we introduce the following market, called $N$-period fractional binary market, which converges to the fractional Black-Scholes model \eqref{fbse} as shown in \cite{Sotti}. Let $(\Omega,\Fs, P)$ be a finite probability space. The bond and the stock are traded at the times $\{0, \frac{1}{N},...,\frac{N-1}{N}, 1\}$ under the dynamics:
$$B_n^{(N)}=1\quad\textrm{and}\quad S_n^{(N,H)}=\left(1+a_n^{(N)}+X_n^{(N,H)}\right)\, S_{n-1}^{(N,H)},\quad n\geq1.$$
The initial values are $B_0^{(N)}=1$ and $S_0^{(N,H)}=s_0$, with $s_0$ a positive constant. Here, $B_n^{(N)}$ and $S_n^{(N,H)}$ are understood to be the value of $B^{(N)}$ and $S^{(N,H)}$ in the time interval $[\frac{n}{N},\frac{n+1}{N})$ for each $n\in\{0,\ldots,N-1\}$. The drift $a^{(N)}$ is used to approximate the continuous drift given in \eqref{fbse} via
$$a_n^{(N)}=\frac{1}{N}a(n/N)$$ 
and therefore, for all $N$,
\begin{equation}\label{ea}
 |a_n^{(N)}|\leq \frac{{||a||}_\infty}{N},\qquad n\in\{1,...,N\}.
\end{equation}
Using the construction of Sottinen, we can express $X_n^{(N,H)}$ as:
\begin{equation}\label{proX}
X_n^{(N,H)}=\sum\limits_{i=1}^{n-1}J_n^{(N,H)}(i)\,\xi_i+g_n^{(N,H)}\,\xi_n.
\end{equation}
The random variables ${(\xi_i)_{i\geq 1}}$ are supposed to be i.i.d.~Bernoulli, i.e.~$P(\xi_1=-1)=P(\xi_1=1)=1/2$. The real numbers $J_n^{(N,H)}(i)$, for $1\leq i< n\leq N$, and $g_n^{(N,H)}$, for $1\leq n\leq N$, are defined below. We endow the probability space with the following filtration $\Fs_n=\sigma(\xi_1,\ldots,\xi_n)$, for $n\geq1$, and $\Fs_0=\{\emptyset,\Omega\}$, and therefore the process $S_n^{(N,H)}$ is adapted. 

Now, we define the constants:
 \begin{equation}\label{d2j}
  J_n^{(N,H)}(i):=\sigma\sqrt{N}\int\limits_{\frac{i-1}{N}}^{\frac{i}{N}}\left(k_H\left(\frac{n}{N},u\right)-k_H\left(\frac{n-1}{N},u\right)\right)du,
 \end{equation}
and
\begin{equation}\label{d2g}
 g_n^{(N,H)}:=\sigma\sqrt{N}\int\limits_{\frac{n-1}{N}}^{\frac{n}{N}}k_H\left(\frac{n}{N},u\right)du,
\end{equation}
where
\begin{equation}\label{dk}
 k_H(t,s):=c_H\left(H-\frac{1}{2}\right) s^{\frac{1}{2}-H}\int\limits_s^t u^{H-\frac{1}{2}} {(u-s)}^{H-\frac{3}{2}}du,
\end{equation}
and $c_H:=\sqrt{\frac{2H\,\Gamma\left(\frac{3}{2}-H\right)}{\Gamma\left(H+\frac{1}{2}\right)\,\Gamma(2-2H)}}$ is a normalizing constant. For simplicity, we will use from time to time the notation $C_H:=c_H\, \left(H-\frac{1}{2}\right)$.
We also define, for each $n\in\{1,...,N\}$, the functions $Y_{n}^{(N,H)}:\{-1,1\}^{n-1}\rightarrow\Rb$ by:
$$Y_{n}^{(N,H)}(x_1,...,x_{n-1}):=\sum\limits_{i=1}^{n-1}J_n^{(N,H)}(i)\,x_i, \quad Y_1^{(N,H)}=0.$$
We shortly denote $Y_{n}^{(N,H)}$ the random variable $Y_{n}^{(N,H)}(\xi_1,...,\xi_{n-1})$. In particular, from \eqref{proX} we have the following identity:
$$X_n^{(N,H)}=Y_{n}^{(N,H)}+g_n^{(N,H)}\,\xi_n,$$
where the first process denotes the contribution of the past (up to time $n-1$) and the second one depends only on the present (at time $n$). 
The following functions on ${\{-1,1\}}^{n-1}$ are introduced:
$$u_n^{(N)}(x_1,...,x_{n-1}):=Y_{n}^{(N,H)}(x_1,...,x_{n-1})+g_n^{(N,H)},$$
and
$$d_n^{(N)}(x_1,...,x_{n-1}):=Y_{n}^{(N,H)}(x_1,...,x_{n-1})-g_n^{(N,H)},$$
with $u_1^{(N)}$ and $d_1^{(N)}$ constants. Thus, given the history up to time $n-1$, the process $X^{(N,H)}$ can take at each time $n$ only two possible values $u_n^{(N)}$ and $d_n^{(N)}$ with $d_n^{(N)}<u_n^{(N)}$. This justifies the binary structure of these markets. 
%%%%%%%%section3%%%%%%%%%%%
\section{Arbitrage points and main results}\label{Sec3}
In this section we introduce the notions of arbitrage points and arbitrage paths and formulate our main results concerning their asymptotic properties.

We know from \cite{Dzh} (see \cite{CRR} for the binomial case) that the $N$-period fractional binary market excludes arbitrage opportunities if and only if for all $n\in\{1,...,N\}$ and $x\in{\{-1,1\}}^{n-1}$, we have:
\begin{equation}\label{nac1}
d_n^{(N)}(x)<-a_n^{(N)}< u_n^{(N)}(x).
\end{equation}
The previous characterization of the arbitrage opportunities in the fractional binary market motivates the next definitions. We call the following set $N$-binary tree:
$$\Xb_N=\{\tau\}\cup\bigcup\limits_{n=1}^{N-1} \{-1,1\}^{n},$$
where $\tau$ denotes the root of the tree. We say that a point $x\in\Xb_N$ is an arbitrage point for the $N$-period fractional binary market if  $x$ does not satisfy condition \eqref{nac1}. More precisely, given a level $n\in\{1,...,N\}$, the following object is called the set of arbitrage points at level $n$:
$$\As_n^{(N,H)}:=\{x\in {\{-1,1\}}^{n-1}: u_n^{(N)}(x)\leq -a_n^{(N)} \textrm{ or } d_n^{(N)}(x)\geq -a_n^{(N)}\},\quad n\geq 2,$$
and $\As_1^{(N,H)}$ is equal to $\{\tau\}$ if $u_1^{(N)}\leq -a_1^{(N)}$ or $d_1^{(N)}\geq -a_1^{(N)}$ and the empty set otherwise. The set of arbitrage points is given by:
$$\As^{(N,H)}:=\bigcup\limits_{n=1}^N\As_n^{(N,H)}\subseteq\Xb_N.$$
In addition, we call arbitrage paths the paths in the $N$-binary tree which cross at least one arbitrage point, i.e. the elements of the set:
$$\As_\Ps^{(N,H)}:=\{(x_1,...,x_{N-1})\in{\{-1,1\}}^{N-1}: \exists n\in\{1,...,N\},\, (x_1,...,x_{n-1})\in \As_n^{(N,H)}\}.$$
\begin{remark}
Sottinen proves in \cite{Sotti} that for $N$ large enough, the $N$-period fractional binary market admits arbitrage opportunities. Indeed, it is proved that there exists $n_H>0$ such that for all $N$ sufficiently large:
$$\As_{n_H}^{(N,H)}\neq\emptyset\quad\textrm{and}\quad\lim\limits_{N\rightarrow\infty}\frac{\arrowvert\As_\Ps^{(N,H)}\arrowvert}{2^N}\geq 2^{2-n_H}>0.$$
\end{remark}

Now, we formulate our first main result. In \cite{Sotti}, Theorem 5, the author shows that starting from the root of the binary tree and going always up we can always reach an arbitrage point. The following proposition provides a generalization of that result, establishing that starting from any point in the binary tree by going always up (or always down) we can always reach an arbitrage point. In what follows we use the notation $1_k:=(1,\ldots,1)\in\Rb^k$. 
\begin{proposition}\label{arbpt}
 For all $k\geq 2$ and $x\in\{-1,1\}^{k-1}$, there exist $n_k(x)\geq 1$ and $N_k(x)\geq k+n_{k}(x)$ such that for all $N\geq N_{k}(x):$
$$\left(x,1_{n_k(x)}\right)\in\As_{k+n_k(x)}^{(N,H)}\,\textrm{ and } \left(x,-1_{n_k(x)}\right)\in\As_{k+n_k(x)}^{(N,H)}.$$
In particular, 
$$\lim_{N\rightarrow\infty}\arrowvert\As^{(N,H)}\arrowvert=\infty.$$
\end{proposition}
Before stating the next main result, we introduce the following notations. For $h:=\frac{H}2+\frac14\in(\frac12,\frac34)$, consider the random variable
\begin{equation}\label{dyh}
 \Ys_H:=2\,g_H\,\sum\limits_{k=1}^\infty \rho_{h}(k)\,\xi_k,
\end{equation}
where $g_H:=\frac{\sigma\, c_H}{H+\frac{1}{2}}$ and
\begin{equation}\label{drh}
 \rho_{h}(k):=\frac{1}{2}\left((k+1)^{2h}+{(k-1)}^{2h}-2{k}^{2h}\right).
\end{equation}
Therefore
\begin{equation*}
 \rho_{h}(k)\sim h(2h-1)\,k^{2h-2}=\frac{1}{8}(2H+1)(2H-1)\,k^{H-\frac{3}{2}}\quad\textrm{as}\quad k\rightarrow\infty.
\end{equation*}
Consequently, $\sum_{k=1}^\infty \rho_{h}(k)=\infty$ and, hence, the series in \eqref{dyh} is not absolutely convergent. However, since $\sum_{k=1}^\infty \rho_{h}^2(k)<\infty$, $\Ys_H$ is well defined in the sense of almost sure convergence (see p. 113 in \cite{Will} or \cite{Kac}). We remark that $\rho_h$ is the autocovariance function of a fractional Brownian motion with Hurst parameter $h$.

Our second main result characterizes the proportion of arbitrage points.
\begin{theorem}\label{proparbpt}
For any sequence $N_n\geq n$:
$$\lim\limits_{n\rightarrow\infty}\frac{|\As_n^{(N_n,H)}|}{2^{n-1}}=P(|\Ys_H|>g_H)>0.$$
In particular:
$$\lim\limits_{N\rightarrow\infty}\frac{|\As^{(N,H)}|}{2^{N}-1}=P(|\Ys_H|>g_H).$$
\end{theorem}

The next proposition provides the behaviour of the previous asymptotic proportion when $H$ is close to $1/2$ and when $H$ is close to $1$.

\begin{proposition}\label{proparbptH}
We have that:
$$\lim\limits_{H\rightarrow\frac{1}{2}}P(|\Ys_H|>g_H)=0
\quad\text{and}\quad\liminf\limits_{H\rightarrow 1}P(|\Ys_H|>g_H)\geq \frac{1}{3}.$$
\end{proposition}
In order to state our last result concerning the asymptotic proportion of arbitrage paths, we recall the following notions. For any sequence of measurable sets $A_1, A_2,...$, we denote $\{A_n\, \textrm{i.o.}\}$ and $\{A_n\, \textrm{ult.}\}$, respectively the sets where $A_n$ happens infinitely often and where $A_n$ happens ultimately, by:
$$\{A_n\, \textrm{i.o.}\}:=\bigcap\limits_{n\geq 1}\bigcup\limits_{k\geq n}A_k\quad\textrm{and}\quad\{A_n \textrm{ult.}\}:=\bigcup\limits_{n\geq 1}\bigcap\limits_{k\geq n}A_k.$$ 
\begin{theorem}\label{asp} There exists $H_c\in(\frac{1}{2},1)$ such that for $H>H_c$
 $$P(|\Ys^H_n|>g_n^H\ \textrm{i.o.})=1$$
 and for $H<H_c$
 $$P(|\Ys^H_n|>g_n^H\ \textrm{ult.})=0.$$
 In particular, if $H>H_c$ then 
 $$\lim_{N\to\infty}\frac{\As_\Ps^{(N,H)}}{2^{N-1}}=1.$$
\end{theorem}

\begin{remark}
 The previous theorem tells us that, when $H>H_c$, with probability $1$ a path in the binary tree asymptotically crosses an infinite number of arbitrage points. On the other hand, when $H<H_c$, with probability $1$ a path in the binary tree asymptotically crosses an infinite number of non-arbitrage points. 
\end{remark}

%%%%%%%%%%%%%section4%%%%%%%%%%%%
\section{A discrete scaling property and the proof of Proposition~\ref{arbpt}}\label{Sec4}
In the next proposition, we prove that the dependence on $N$ of the coefficients in the definition of $X_n^{(N,H)}$ appears as a multiplicative scaling factor.
\begin{proposition}\label{sp}
The following statements hold:\\
$(1)$ For all $1\leq i< n\leq N$, we have:
 $$J_n^{(N,H)}(i)=\frac{1}{N^H}\, j_n^H(i),$$
 where
 \begin{equation}\label{fj}
 j_n^H(i):=\sigma\, C_H \int\limits_{i-1}^{i}x^{\frac{1}{2}-H}\left(\int\limits_0^1 (v+n-1)^{H-\frac{1}{2}} (v+n-1-x)^{H-\frac{3}{2}}dv\right) dx,
\end{equation}
$(2)$ For all $1\leq n\leq N$, we have:
 $$g_n^{(N,H)}=\frac{1}{N^H} \, g_n^H,$$
where
\begin{equation}\label{fg}
 g_n^H:=\sigma\, C_H\int\limits_{n-1}^{n}x^{\frac{1}{2}-H}(n-x)^{H-\frac{1}{2}}\left(\int\limits_0^1 (y(n-x)+x)^{H-\frac{1}{2}}y^{H-\frac{3}{2}}dy\right)dx.
\end{equation}
In particular, for all $1\leq n\leq N$, we have that:
\begin{equation*}
N^H\, X_n^{(N,H)}= n^H\,X_n^{(n,H)}.
\end{equation*}

\end{proposition}
\begin{proof}
(1) From equation \eqref{dk}, we have that
$$k_H\left(\frac{n}{N},u\right)-k_H\left(\frac{n-1}{N},u\right)=C_H \,u^{\frac{1}{2}-H}\int\limits_{\frac{n-1}{N}}^{\frac{n}{N}}s^{H-\frac{1}{2}}(s-u)^{H-\frac{3}{2}}ds.$$
By means of the change of variable $v=Ns-n+1$, the last identity implies that:
$$k_H\left(\frac{n}{N},u\right)-k_H\left(\frac{n-1}{N},u\right)=\frac{C_H \,u^{\frac{1}{2}-H}}{N^{2H-1}}\int\limits_{0}^{1}(v+n-1)^{H-\frac{1}{2}}(v+n-1-Nu)^{H-\frac{3}{2}}dv.$$
The result follows by plugging this expression in the definition of $J_n^{(N,H)}(i)$, see \eqref{d2j}, and making the change of variable $x=Nu$.

(2) Using the definition of $g_n^{(N,H)}$ in \eqref{d2g} and the change of variable $x=Nu$, it follows
  \begin{equation*}
  g_n^{(N,H)}=\frac{\sigma}{\sqrt{N}}\int\limits_{n-1}^{n}k_H\left(\frac{n}{N},\frac{x}{N}\right)dx.
 \end{equation*}
On the other hand, we have that:
\begin{equation}\label{d2k}
 k_H\left(\frac{n}{N},\frac{x}{N}\right)=C_H N^{H-\frac{1}{2}} x^{\frac{1}{2}-H}\int\limits_{\frac{x}{N}}^{\frac{n}{N}}s^{H-\frac{1}{2}}\left(s-\frac{x}{N}\right)^{H-\frac{3}{2}}ds.
\end{equation}
By means of the change of variable $Ns=y(n-x)+x$, the integral in the previous identity can be expressed in the following form:
\begin{equation*}
 \int\limits_{\frac{x}{N}}^{\frac{n}{N}}s^{H-\frac{1}{2}}\left(s-\frac{x}{N}\right)^{H-\frac{3}{2}}ds=\frac{1}{N^{2H-1}}(n-x)^{H-\frac{1}{2}}\int\limits_0^1 (y(n-x)+x)^{H-\frac{1}{2}}y^{H-\frac{3}{2}}dy.
 \end{equation*}
Plugging the last expression in \eqref{d2k}, and using the resulting identity in \eqref{d2g}, we obtain the desired result.
\end{proof}
Inspired by the previous proposition, we define the random variables $\Ys^H_n$ as:
$$\Ys^H_n:=\sum\limits_{i=1}^{n-1}j_n^H(i)\,\xi_i.$$
From Proposition \ref{sp}, for all $1\leq n\leq N$, the following identities hold:
\begin{equation}\label{spcys}
 Y_n^{(N,H)}=\frac{1}{N^H}\,\Ys^H_n,\qquad X_n^{(N,H)}=\,\frac{1}{N^H}(\Ys^H_n+g_n^H\,\xi_n).
\end{equation}
The proof of Proposition \ref{arbpt} requires upper and lower bounds for the quantities  $J_n^{(N,H)}(i)$ and $g_n^{(N,H)}$. Thanks to the above-mentioned scaling property, it is enough to bound the parameters $j_n^H(i)$ and $g_n^H$, which is done in the next Lemma.
\begin{lemma}\label{ejg}The following inequalities hold:\\
$(1)$ For all $1\leq i\leq n-1< N$, we have:
\begin{equation}\label{ej1}
 \sigma\, c_H \,(n-1)^{H-\frac{1}{2}}\,I_n(i)\leq j_n^H(i)\leq \sigma\, c_H\, n^{H-\frac{1}{2}}\,I_n(i),
\end{equation}
where
$$I_n(i):=\int\limits_{i-1}^i x^{\frac{1}{2}-H}\phi_n^H(x)dx\quad \textrm{and}\quad \phi_n^H(x):=(n-x)^{H-\frac{1}{2}} -(n-1-x)^{H-\frac{1}{2}}.$$
$(2)$  For all $1< n\leq N$, we have:
\begin{equation}\label{eg1}
g_H\,\leq g_n^H\leq g_H \,\left(1+\frac{1}{n-1}\right)^{H-\frac{1}{2}}.
\end{equation}
 \end{lemma}
 \begin{proof}
(1) Since, for every $v\in[0,1]$, we have $n-1\leq v+n-1\leq n$, we deduce that:
 \begin{align*}
(n-1)^{H-\frac{1}{2}} \frac{\phi_n^H(x) }{H-\frac{1}{2}}&\leq \int\limits_0^1 (v+n-1)^{H-\frac{1}{2}} (v+n-1-x)^{H-\frac{3}{2}}dv \leq  n^{H-\frac{1}{2}}\frac{\phi_n^H(x) }{H-\frac{1}{2}}.
 \end{align*}
 The result is  obtained by plugging the previous inequalities in \eqref{fj}.\\
 (2) Note first that for every $x\in(n-1,n)$ we have:
 $$\frac{x^{H-\frac{1}{2}}}{H-\frac{1}{2}}\leq\int\limits_0^1 (y(n-x)+x)^{H-\frac{1}{2}}y^{H-\frac{3}{2}}dy\leq \frac{n^{H-\frac{1}{2}}}{H-\frac{1}{2}}$$
 Using these inequalities and \eqref{fg}, we obtain the following sequence of inequalities:
 \begin{align*}
  g_H \,\leq g_n^H &\leq \sigma\, c_H n^{H-\frac{1}{2}}\int\limits_{n-1}^{n}x^{\frac{1}{2}-H}(n-x)^{H-\frac{1}{2}}dx \leq  g_H\, n^{H-\frac{1}{2}} (n-1)^{\frac{1}{2}-H},
 \end{align*}
which proves the desired inequalities.
 \end{proof}
 Now, we have all the ingredients for the proof of our first main result.
\begin{proof}[Proof of Proposition~\ref{arbpt}]
 Fix $k\geq 2$. We prove only the first statement. The second one follows analogously. Note that, it is enough to show the result for $x=-1_{k-1}$. More precisely, we prove that, for $n$ sufficiently large, $d_{n+k}^{N}(-1_{k-1},1_n)\geq -a_{k+n}^{(N)}$, which is equivalent to: 
\begin{equation*}
 R_n^N(k):=a_{k+n}^{(N)}+Y_{k+n}^{(N,H)}(-1_{k-1},1_n)-g_{n+k}^{(N,H)}\geq 0.
\end{equation*}
The first term is bounded as in \eqref{ea}. For the last term, we use Proposition \ref{sp} and equation \eqref{eg1} to obtain
\begin{equation}\label{eg}
g_{k+n}^{(N,H)}\leq \frac{c_g}{N^H},
\end{equation}
where $c_g$ is a positive constant.
It remains to obtain a lower bound for the term $Y_{k+n}^{(N,H)}(-1_{k-1},1_n)$. Note first that:
$$Y_{k+n}^{(N,H)}(-1_{k-1},1_n)=\frac{1}{N^H}\left(-\sum\limits_{i=1}^{k-1} j_{k+n}^H(i)+\sum\limits_{i=k}^{k+n-1} j_{k+n}^H(i)\right).$$

Using the upper bound in \eqref{ej1} for $j_{k+n}(i)$, we obtain:
\begin{equation*}
 \sum\limits_{i=1}^{k-1} j_{k+n}^H(i)\leq\sigma\, c_H\, (n+k)^{\alpha}\int\limits_0^{k-1}x^{-\alpha}\phi_{n+k}^H(x)\,dx,
 \end{equation*}
 where $\alpha=H-\frac12\in(0,\frac12)$.
Using the definition of the function $\phi_{n+k}^H$ and some appropriate change of variables, we obtain:

\begin{align*}
\int\limits_0^{k-1}x^{-\alpha}\phi_{n+k}^H(x)\,dx&=(n+k)\int\limits_0^{\frac{k-1}{n+k}}(1-v)^{\alpha} v^{-\alpha}dv-(n+k-1)\int\limits_0^{\frac{k-1}{n+k-1}}(1-v)^{\alpha} v^{-\alpha}dv\\
 &\leq \int\limits_0^{\frac{k-1}{n+k-1}}(1-v)^{\alpha} v^{-\alpha}dv \leq \frac{1}{1-\alpha}\left(\frac{k-1}{n+k-1}\right)^{1-\alpha}.
\end{align*}
Thus, for $n\geq k$:
\begin{equation}\label{ef1}
  \sum\limits_{i=1}^{k-1} j_{k+n}^H(i)\leq \frac{2\,\sigma\,c_H}{1-\alpha}\,n^{\alpha}\left(\frac{k-1}{n+k-1}\right)^{1-\alpha}.
\end{equation}
Now, using the lower bound in \eqref{ej1} for $j_{k+n}(i)$, we have:
\begin{align*}
\sum\limits_{i=k}^{k+n-1} j_{k+n}^H(i)&\geq \sigma\, c_H\, (n+k-1)^{\alpha}\int\limits_{k-1}^{n+k-1}x^{-\alpha}\phi_{n+k}^H(x)\,dx
\end{align*}
Proceeding as before, using an appropriate change of variables, we deduce that:
\begin{align*}
\int\limits_{k-1}^{n+k-1}\!\!x^{-\alpha}\phi_{n+k}^H(x)\,dx&=(n+k)\!\!\!\!\int\limits_{\frac{k-1}{n+k}}^{\frac{n+k-1}{n+k}}\!\!\!\!(1-v)^{\alpha} v^{-\alpha} dv-(n+k-1)\!\!\!\!\int\limits_{\frac{k-1}{n+k-1}}^{1}\!\!\!\!(1-v)^{\alpha} v^{-\alpha} dv\\
&\geq\int\limits_{\frac{k-1}{n+k-1}}^{1}(1-v)^{\alpha} v^{-\alpha}dv-(n+k)\int\limits_{1-\frac{1}{n+k}}^1(1-v)^{\alpha} v^{-\alpha} dv\\
 &\geq \frac{1}{1+\alpha}\left(\left(\frac{n}{n+k-1}\right)^{1+\alpha}-\frac{1}{(n+k-1)^{\alpha}}\right).
\end{align*}
and then, for $n\geq k$ big enough:
\begin{equation}\label{ef2}
 \sum\limits_{i=k}^{k+n-1} j_{n+k}^H(i)\geq \frac{\sigma\,c_H}{4(1+\alpha)}n^{\alpha} .
\end{equation}
Now, using \eqref{ea}, \eqref{eg}, \eqref{ef1} and \eqref{ef2}, we obtain for $n$ big enough:
\begin{align*}
N^H R_n^N(k)\geq\sigma\, c_H\, n^{\alpha}\left(\frac{1}{4(1+\alpha)}-\frac{2}{1-\alpha}\left(\frac{k-1}{n+k-1}\right)^{1-\alpha}\right)-c_g-\frac{||a||_{\infty}}{N^{1-H}}.
\end{align*}
As a consequence, for $n$ and $N$ large enough, $R_n^N(k)\geq 0$, which proves the result.
\end{proof}

%%%%%%%%%section5%%%%%%%%%%%%%%%%%%%%%%%%%%%%%%%%%%%%%%%%%%%%%%%%%%%%%%%%%%%%%%%%%%%%%%%%%%%%%%%%%%%%%%%
\section{On the proportion of arbitrage points and the proofs of Theorem~\ref{proparbpt} and Proposition~\ref{proparbptH}}\label{Sec5}
In this section we identify the asymptotic behaviour of the proportion of arbitrage points with the convergence of a well-chosen sequence of random variables.

From the definition of the set $\As_n^{(N,H)}$, we have:
\begin{align*}
 (x_1,...,x_{n-1})\in\As_n^{(N,H)}&\Leftrightarrow Y_{n}^{(N,H)}(x_1,...,x_{n-1})\notin \left(-g_n^{(N,H)}-a_n^{(N)},g_n^{(N,H)}-a_n^{(N)}\right)\\
 &\Leftrightarrow |Y_{n}^{(N,H)}(x_1,...,x_{n-1})+a_n^{(N)}|\geq g_n^{(N,H)}.
\end{align*}
Since the paths in $\{-1,1\}^{n-1}$ are equidistributed, we have that:
\begin{equation}\label{eap}
\frac{|\As_n^{(N,H)}|}{2^{n-1}}=P\left(|Y_n^{(N,H)}+a_n^{(N)}|\geq g_n^{(N,H)}\right). 
\end{equation}
In a similar way, we can see that 
\begin{equation}\label{ap}
\frac{|\As_\Ps^{(N,H)}|}{2^{N-1}}=P\left(\exists\ n\in\{1,\ldots,N\}:\ |Y_n^{(N,H)}+a_n^{(N)}|\geq g_n^{(N,H)}\right). 
\end{equation}

Thanks to \eqref{eap} and \eqref{spcys}, the proportion of arbitrage points at asymptotic levels is related to the limit behaviour of the random variables $(\Ys^H_n)_{n\geq 1}$. More precisely, for any strictly increasing sequence of positive integers $N_n$, we have the following relation:
\begin{equation}\label{aps}
\frac{|\As_{n}^{(N_n,H)}|}{2^{n-1}}=P\left(|\Ys^H_n+a_n^{(N_n)} N_n^H|\geq g_n^H\right).                                                                                                                                                                                                                                                                                                                                                                                                                                                                                                                                                                                                                                                       \end{equation}
For each $n\geq1$, $\Ys^H_n$ is a sum of independent random variables and its variance is given by
$$Var(\Ys^H_n)=\sum\limits_{i=1}^{n-1}\left(j_n^H(i)\right)^2.$$
However, we cannot apply a CLT in order to study the limit behaviour $\Ys^H_n$. The reason is that, by inequality \eqref{ej1} and the definition of $g_H$:
\begin{equation*}
\lim\limits_{n\rightarrow\infty}j_n^H(n-1)=g_H\left(2^{H+\frac{1}{2}}-2\right)>0,
\end{equation*}
which implies that the Lindeberg condition is not satisfied.
Instead, we express the random variable $\Ys^H_n$ as a sum of two independent random variables $\bar{\Ys}_n^{H}$ and $\widehat{\Ys}^{H}_n$ with very different properties. We do so following the monotonicity properties of the function $I_n$ defined in Lemma \ref{ejg}. Indeed, by Corollary \ref{iin}, there is $i_n\to\infty$ such that $I_n$ is decreasing on $\{1,..,i_n-1\}$ and increasing on $\{i_n+1,...,n-1\}$. This allows to write $\Ys^H_n$ as $\bar{\Ys}_n^{H}+\widehat{\Ys}^{H}_n$, where:
$$\bar{\Ys}_n^{H}=\sum\limits_{i=1}^{i_n-1}j_n^H(i)\,\xi_i\quad \textrm{and}\quad \widehat{\Ys}^{H}_n=\sum\limits_{i=i_n}^{n-1}j_n^H(i)\,\xi_i.$$ 
These random variables are clearly independent and symmetric. In the next sections, their convergence properties are studied.

\subsection{\texorpdfstring{On the random variables $\bar{\Ys}_n^{H}$}{}}\label{bar}
The next result gives the convergence of the first part of the random walk.
\begin{proposition}\label{sum1}
We have that
$\bar{\Ys}_n^{H}\xrightarrow[n\rightarrow\infty]{L^2}0.$
\end{proposition}
\begin{proof}
Note first that, since the function $\phi_n^H$ is increasing, we have:
$$ I_n(1)\leq \frac{(n-2)^{H-\frac12}}{\frac32-H}\left[\left(1+\frac{1}{n-2}\right)^{H-\frac12} -1\right].$$
Plugging this in \eqref{ej1}, we get
\begin{equation}\label{cjn1}
j_n(1)\xrightarrow[n\rightarrow\infty]{}0. 
\end{equation}
Similarly, for $1<i\leq i_n-1$, we deduce that:
$$I_n(i)\leq \frac{1}{(i-1)^{H-\frac12}}\phi_n^H(i_n-1).$$
Thus, using \eqref{ej1}, we obtain:
$$(j_n^H(i))^2\leq (\sigma c_H)^2  \frac{n^{2H-1}}{(i-1)^{2H-1}}\left({\phi_n^H}(i_n-1)\right)^2.$$
Moreover, from Lemma \ref{mgn} and Corollary \ref{iin}, we see that, for $n$ sufficiently large, $i_n-1\leq (2H-1)(n-1)$, and hence:
$$(j_n^H(i))^2\leq (\sigma c_H)^2\frac{n^{2H-1}}{(i-1)^{2H-1}}\left({\phi_n^H}\left((2H-1)(n-1)\right)\right)^2.$$
In addition, using the definition of $\phi_n^H$, we see that
$$\phi_n^H\left((2H-1)(n-1)\right)\leq n^{H-\frac12}\left[\left(1+\frac{1}{(2-2H)(n-1)}\right)^{H-\frac12}-1\right].$$
Therefore, we can choose $c_H^*>0$ such that for any $n$ large enough:
$$(j_n^H(i))^2\leq \frac{c_H^*}{(i-1)^{2H-1} \,n^{4-4H}}.$$
On the other hand, since:
$$\sum\limits_{i=1}^{\infty}\frac{1}{i^{2H-1} \,i^{4-4H}}=\sum\limits_{i=1}^{\infty}\frac{1}{i^{3-2H}}<\infty,$$
by Kronecker's lemma (see Lemma 4.21 in \cite{Kall}), we conclude that:
$$\frac{1}{n^{4-4H}}\sum\limits_{i=1}^{i_n-1}\frac{1}{i^{2H-1}}\xrightarrow[n\rightarrow\infty]{}0.$$
This together with \eqref{cjn1} implies that
$$\Var(\bar{\Ys}_n^{H})\leq j_n^H(1)^2+\frac{c_H^*}{n^{4-4H}}\sum\limits_{i=2}^{i_n-1}\frac{1}{(i-1)^{2H-1}}\xrightarrow[n\rightarrow\infty]{}0.$$
The result is proved.
\end{proof}

\subsection{\texorpdfstring{On the random variables $(\widehat{\Ys}^{H}_n)_{n\geq 1}$}{}}\label{hat}

We start this section, proving that the variance of $\widehat{\Ys}^{H}_n$ converges.
\begin{lemma}\label{cov}
We have that:
$$\Var(\widehat{\Ys}^{H}_n)\xrightarrow[n\rightarrow\infty]{}4g_H^2\,\sum\limits_{k=1}^\infty \rho_{h}^2(k)<\infty.$$
where $h=\frac{H}{2}+\frac{1}{4}\in(\frac{1}{2},\frac{3}{4})$ and $\rho_h$ is defined in \eqref{drh}.
\end{lemma}
\begin{proof}
Note first that:
$$\Var(\widehat{\Ys}^{H}_n)=\sum\limits_{i=i_n}^{n-1}(j_n^H(i))^2=\sum\limits_{k=1}^{n-i_n}(j_n^H(n-k))^2.$$
On the other hand, using \eqref{ej1}, we obtain for $1\leq k\leq n-i_n$:
$$j_n^H(n-k)\leq \sigma\, c_H\left(\frac{n}{n-k-1}\right)^{H-\frac12}\left(\frac{(k+1)^{H+\frac12}+{(k-1)}^{H+\frac12}-2{k}^{H+\frac12}}{H+\frac12}\right)$$
and
$$j_n^H(n-k)\geq \sigma\, c_H\left(\frac{n-1}{n-k}\right)^{H-\frac12}\left(\frac{(k+1)^{H+\frac12}+{(k-1)}^{H+\frac12}-2{k}^{H+\frac12}}{H+\frac12}\right).$$
It follows that, for any $k\geq 1$:
$$\lim\limits_{n\rightarrow\infty}j_n^H(n-k)=\sigma\, c_H\left(\frac{(k+1)^{H+\frac12}+{(k-1)}^{H+\frac12}-2{k}^{H+\frac12}}{H+\frac12}\right)=2g_H\,\rho_h(k).$$
In the same way, using Lemma \ref{mgn}, Corollary \ref{iin} and the previous upper bound for $j_n^H(n-k)$, one can find a constant $M>0$, such that for any $n$ sufficiently large:
$$j_n^H(n-k)\leq M\,\rho_h(k).$$
Finally, since $\sum_{k=1}^\infty \rho_{h}^2(k)<\infty$, the desired result follows as an application of the dominated convergence theorem.
\end{proof}

Motivated by Lemma \ref{cov}, we introduce the random sequence $(\Upsilon^H_n)_{n\geq1}$:
$$\Upsilon^H_n:=\sum\limits_{k=1}^{n-i_n}j_n^H(n-k)\,\xi_k,\quad n\in\Nb.$$
It is clear that $\Upsilon^H_n\overset{d}{=}\widehat{\Ys}^{H}_n$. The next proposition reinforces Lemma \ref{cov}, with the help of the random variables $\Upsilon^H_n$, and permits to conclude the convergence in law of the random variables $\widehat{\Ys}^{H}_n$. 
\begin{proposition}\label{sum2}
We have that $\Upsilon^H_n\xrightarrow[n\rightarrow\infty]{L^2}\Ys_H$ and therefore $\widehat{\Ys}^{H}_n\xrightarrow[n\rightarrow\infty]{d}\Ys_H$, where $\Ys_H$ is defined in \eqref{dyh}.
\end{proposition}
\begin{proof}
Let's denote $k_n=n-i_n$. It is straightforward to see that:
\begin{align*}
 E[(\Upsilon^H_n-\Ys_H)^2]&=E\left[\left(\sum\limits_{k=1}^{k_n}\left(j_n^H(n-k)-2g_H\rho_{h}(k)\right)\,\xi_k-2g_H\sum\limits_{k=k_n+1}^{\infty}\rho_{h}(k)\xi_k\right)^2\right]\\
 &=\sum\limits_{k=1}^{k_n}\left(j_n^H(n-k)-2g_H\rho_{h}(k)\right)^2+4g_H^2\sum\limits_{k=k_n+1}^{\infty}\rho_{h}^2(k).
\end{align*}
Note that the convergence to $0$ of the second sum is guaranteed as $\sum_{k=1}^\infty \rho_{h}^2(k)<\infty.$ 
It remains to show that the first sum on the right-hand side also converges to $0$. For this, using the bounds obtained in the proof of the Lemma \eqref{cov}, we see that for each $k\leq k_n$:
$$j_n^H(n-k)-2g_H\rho_{h}(k)\xrightarrow[n\rightarrow\infty]{}0$$
and $|j_n^H(n-k)-2g_H\rho_{h}(k)|\leq C\rho_{h}(k) $ for some constant $C>0$. Therefore, the result is obtained as an application of the dominated convergence theorem.

For the second statement, we use that convergence in $L^2$ implies convergence in law and that $\Upsilon^H_n\overset{d}{=}\widehat{\Ys}^{H}_n$. 
\end{proof}

\subsection{\texorpdfstring{On the convergence of $(\Ys^H_n)_{n\geq 1}$ and applications}{}}
Now we have all the necessary elements to establish the convergence of the random variables $(\Ys^H_n)_{n\geq 1}$, which is provided in the next theorem.
\begin{theorem}
We have that $\Ys^H_n\xrightarrow[n\rightarrow\infty]{d}\Ys_H.$
\end{theorem}
\begin{proof}
Proposition~\ref{sum1} implies that $\bar{\Ys}_n^{H}\xrightarrow[n\rightarrow\infty]{p}0$, and  since, by Proposition~\ref{sum2}, $\widehat{\Ys}^{H}_n\xrightarrow[n\rightarrow\infty]{d}\Ys_H$ we obtain by Slutski's theorem that
$$\Ys^H_n=\bar{\Ys}_n^{H}+\widehat{\Ys}^{H}_n\xrightarrow[n\rightarrow\infty]{d}\Ys_H.$$
\end{proof}

\begin{theorem}
The law of $\Ys_H$ is absolutely continuous with respect to the Lebesgue measure, its density $f_H$ is symmetric, bounded, $L^2(\Rb)$-integrable and has non compact support. 
\end{theorem}
\begin{proof} We claim that the characteristic function of $\Ys_H$, $F_H(v)=E[e^{i v \Ys_H}]$, decays faster than exponentially. If this is true, then $F_H$ is in $L^2(\Rb)$ and  the law of $\Ys_H$ admits a density function $f_H$ in $L^2(\Rb)$ (Lemma 2.1 in \cite{Ber}). The relation between the $L^2$-norms of $F_H$ and $f_H$ is given by the Plancherel's theorem. We can also deduce that $F_H$ is in $L^1(\Rb)$, which implies that $f_H$ is bounded (Corollary 5.1, Chapter 9 in \cite{Res}). The fact that $f_H$ is symmetric comes from the symmetry of the law of $\Ys_H$. The last assertion is a consequence of the uncertainty principle, which informally asserts that $F_H$ and $f_H$ cannot both decay too fast at infinity (see for example \cite{Horm}).

Now, we turn to the proof of the claim. Since the series defining $\Ys_H$ is almost surely convergent (and therefore in distribution), we deduce that:
\begin{align*}
 E[e^{iv\Ys_H}]&=\lim\limits_{n\rightarrow\infty}E\left[\exp\left(2iv g_H\sum\limits_{k=1}^{n}\rho_h(k)\,\xi_k\right)\right]\\&=\lim\limits_{n\rightarrow\infty}E\left[\prod\limits_{k=1}^{n}\exp\left(2iv g_H\rho_h(k)\,\xi_k\right)\right]
 =\lim\limits_{n\rightarrow\infty}\prod\limits_{k=1}^{n}E\left[\exp\left(2iv g_H\rho_h(k)\,\xi_k\right)\right]\\&=\lim\limits_{n\rightarrow\infty}\prod\limits_{k=1}^{n}\cos\left(2v g_H\rho_h(k)\right)=\prod\limits_{k=1}^{\infty}\cos\left(2v g_H\rho_h(k)\right).
\end{align*}
We obtain first bounds for $\cos(u \rho_H(k))$. We assert that for any $x\in(0,\pi/2)$:
$$0<\cos(x)\leq 1-\frac{x^2}{\pi}.$$
In order to prove that, we consider the function $f$ defined by $f(x)=1-\frac{x^2}{\pi}-\cos(x)$. Since $f(0)=0$, it would be enough to prove that $f$ is increasing in $(0,\pi/2)$. This is indeed the case, as for each $x\in(0,\pi/2)$:
$$f'(x)=-\frac{2x}{\pi}+\sin(x)\geq 0,$$
which proves our assertion (the last inequality follows from the concavity of the sinus function on $(0,\pi/2)$).

On the other hand, since $\rho_h(k)\sim h(2h-1)k^{2h-2}$ when $k$ goes to infinity, we can find $k_0$ and $b_H>\gamma_H>0$ such that, for any $k\geq k_0$:
$$\gamma_H\, k^{-\beta} \leq\rho_h(k)\leq b_H\, k^{-\beta}$$
where $\beta=2-2h\in(1/2,1)$. Now, for each $u>0$, we define:
$$k(u)=k_0\vee\inf\left\{k\in\Nb: \frac{u\, b_H}{k^\beta}\leq \frac{\pi}{2}\right\}.$$
From the definition, we have that for any $k\geq k(u)$:
$$0<\frac{u\, \gamma_H}{k^\beta}\leq u\,\rho_h(k)\leq \frac{u\, b_H}{k^\beta}\leq \frac{\pi}{2}.$$
In particular,
$$\prod\limits_{k=1}^{\infty}|\cos(u \rho_h(k))|\leq \prod\limits_{k=k(u)}^{\infty}|\cos(u \rho_h(k))|\leq \prod\limits_{k=k(u)}^{\infty}\left(1-\frac{(u\rho_h(k))^2}{\pi}\right).$$
and then:
$$\prod\limits_{k=1}^{\infty}|\cos(u \rho_h(k))|\leq \prod\limits_{k=k(u)}^{\infty}\left(1-\frac{(u\,\gamma_H)^2}{\pi k^{2\beta}}\right)$$
On the other hand:
$$\ln\left(\prod\limits_{k=k(u)}^{\infty}\left(1-\frac{(u\,\gamma_H)^2}{\pi k^{2\beta}}\right)\right)=\sum\limits_{k=k(u)}^{\infty}\ln\left(1-\frac{(u\,\gamma_H)^2}{\pi k^{2\beta}}\right).$$
Note that for $u$ big enough, $k(u)>k_0$ and hence $k(u)=\inf\{k\in\Nb: \frac{u\, b_H}{k^\beta}\leq \frac{\pi}{2}\}$. This implies that $(k(u)-1)^\beta\leq \frac{2ub_H}{\pi}\leq {k(u)}^\beta$ and then: 
\begin{align*}
 \sum\limits_{k=k(u)}^{\infty}\ln\left(1-\frac{(u\,\gamma_H)^2}{\pi k^{2\beta}}\right)&\leq \sum\limits_{k=k(u)}^{\infty}\ln\left(1-\frac{\pi\gamma_H^2}{(2b_H)^2 }\left(\frac{k(u)-1}{k}\right)^{2\beta}\right)\\
 &\leq  (k(u)-1)\int\limits_1^{\infty}\ln\left(1-\frac{\pi\gamma_H^2}{(2b_H)^2 x^{2\beta}}\right)dx\\
 &\leq -\theta_H u^{1/\beta},
\end{align*}
for some constant $\theta_H >0$ and $u>0$ sufficiently large. Consequently, setting $u=2g_H |v|$ and using the symmetry of $f_H$, we deduce that for $|v|$ sufficiently large: 
$$|E[e^{iv\Ys_H}]|\leq e^{-\theta_H (2g_H\,|v|)^{1/\beta}}.$$
The claim is then proved.
\end{proof}

\begin{proof}[Proof of Theorem~\ref{proparbpt}]
Using that $\lim_{n\rightarrow\infty}g_n^H=g_H$, $\lim_{n\rightarrow\infty}\frac{a_n^{(N_n)} N_n^H}{g_n^H}=0$ and $\Ys^H_n\xrightarrow[n\rightarrow\infty]{d}\Ys_H$, we get by Slutzky's theorem that
$$\frac{\Ys^H_n}{g_n^H}+\frac{a_n^{(N_n)} N_n^H}{g_n^H}\xrightarrow[n\rightarrow\infty]{d}\frac{\Ys_H}{g_H}.$$
Since the law of $\Ys_H$ is absolutely continuous, we apply the Portmanteau theorem and Lemma 17.2 of \cite{Will} to deduce that
\begin{align*}
 \lim\limits_{n\rightarrow\infty}P\left(|\Ys^H_n+a_n^{(N_n)} N_n^H|\geq g_n^H\right)&=\lim\limits_{n\rightarrow\infty}P\left(\frac{|\Ys^H_n+a_n^{(N_n)} N_n^H|}{g_n^H}\geq 1\right)\\
 &=P\left(\frac{|\Ys_H|}{g_H}>1\right)=P(|\Ys_H|>g_H).
\end{align*}
The proof of the first statement is achieved using \eqref{aps} and the fact that the density $f_H$ has no compact support.

For the second statement in the theorem, first note that:
\begin{equation}\label{decomp}
  |\As^{(N,H)}|=\sum\limits_{n=1}^N|\As_n^{(N,H)}|=\sum\limits_{n=1}^NP(|\Ys_n^H + a_n^{(N)} N^H|\geq g_n^H)\, 2^{n-1}.
\end{equation}
Now, fix $\varepsilon>0$ and consider $N$ sufficiently large in order to satisfy $||a||_\infty /N^{1-H}\leq \varepsilon$. For such $N$ and $n\leq N$, we see that:
$$P(|\Ys_n^H|\geq g_n^H +\varepsilon) \leq P(|\Ys_n^H + a_n^{(N)} N^H|\geq g_n^H)\leq P(|\Ys_n^H|\geq g_n^H -\varepsilon),$$
and then, plugging this in \eqref{decomp}, we get: 
\begin{equation}\label{ineq}
\sum\limits_{n=1}^NP(|\Ys_n^H|\geq g_n^H +\varepsilon)\, 2^{n-1}\leq|\As^{(N,H)}|\leq \sum\limits_{n=1}^NP(|\Ys_n^H|\geq g_n^H -\varepsilon)\, 2^{n-1}.
\end{equation}
Additionally, as in the proof of the first statement above, we get:
$$P(|\Ys_n^H|\geq g_n^H +\varepsilon)\xrightarrow[n\rightarrow\infty]{}P(|\Ys_H|> g_H +\varepsilon),$$
and
$$P(|\Ys_n^H|\geq g_n^H -\varepsilon)\xrightarrow[n\rightarrow\infty]{}P(|\Ys_H|> g_H -\varepsilon).$$
Applying C\`{e}saro's lemma (page 116 in \cite{Will}) to the sequences $\{P(|\Ys_n^H|\geq g_n^H +\varepsilon)\}_{n\geq 1}$ and $\{P(|\Ys_n^H|\geq g_n^H-\varepsilon)\}_{n\geq 1}$, \eqref{ineq} leads to:
\begin{align*}
P(|\Ys_H|> g_H +\varepsilon)&\leq\liminf_{N\rightarrow\infty}\frac{|\As^{(N,H)}|}{2^N-1}\leq\limsup_{N\rightarrow\infty}\frac{|\As^{(N,H)}|}{2^N-1}\leq P(|\Ys_H|> g_H -\varepsilon).
\end{align*}
The result follows by taking the limit when $\varepsilon$ tends to $0$.
\end{proof}
\begin{proof}[Proof of Proposition~\ref{proparbptH}]
Thanks to the Tchebysheff's inequality, we obtain that:
$$P(|\Ys_H|>g_H)\leq 4 \sum\limits_{k=1}^\infty\rho_h^2(k).$$
In addition, from Lemma \ref{mono}, we have that $\lim_{h\rightarrow\frac{1}{2}+}\sum_{k=1}^\infty\rho_h^2(k)=0,$
and, since $h=\frac{H}{2}+\frac14$, the first result follows.\\
From Corollary \ref{hc}, it follows that for $H>2h_c-1/2$, we have $\sum_{k=1}^\infty\rho_h^2(k)>\frac{1}{4}$. Hence, we can use the Paley-Zigmund inequality (Lemma 4.1 in \cite{Kall}) and a particular case of the Khintchine's inequality (see \cite{Khin}), to obtain:
$$P(|\Ys_H|>g_H)\geq \frac{1}{3}\left(1-\left(\frac{1}{4\sum\limits_{k=1}^\infty\rho_h^2(k)}\right)\right)^2.$$
On the other hand, we know from Lemma \ref{lbr} that $\lim_{h\rightarrow \frac34-}\sum_{k=1}^\infty\rho_h^2(k)=\infty.$
Combining this with the previous inequality, we obtain the desired result.
\end{proof}
%%%%%%%%%%%%%%%%%%%%%%%%%%%%%%%%%%%%%%%%%%%%%%%%%%%%%%%%%%%%%%%%%%%%%%%%%%%%%%%%%%%%%%%%%%%%%%%%
%%%%%%%%%%%%%%%%%%%%%%%%%%%%%%%%%%%%%%%%%%%%%%%%%%%%%%%%%%%%%%%%%%%%%%%%%%%%%%%%%%%%%%%%%%%%%%%%
%%%%%%%section6%%%%%%%%%%%%%%%%%%%%%%%%%%%%%%%%%%%%%%%%%%%%%%%%%%%%%%%%%%%%%%%%%%%%%%%%%%%%%%%%%%%%%%%%%
\section{Proof of Theorem \ref{asp}}\label{Sec6}
In this section, we exploit the measurability properties of the random variables $\widehat{\Ys}^{H}_n$. Intuitively, these random variables depend asymptotically only on the tail $\sigma$-field, which is defined by:
$$\Ts_\infty=\bigcap\limits_{n\geq 1}\bigvee\limits_{k>n}\sigma(\xi_k)=\bigcap\limits_{n\geq 1}\sigma(\xi_k, k> n).$$
Since the random variables $(\xi_i)_{i\geq1}$ are independent, we know from the Kolmogorov 0-1 law (see Theorem 3.13 in \cite{Kall}) that $\Ts_\infty$  is $P$-trivial and that the $\Ts_\infty$-measurable random variables are constant. One could be tempted to say that $\Ys_H$ is then constant, which is in contradiction with the fact that its variance is strictly positive. This contradiction is only apparent and the reason is that the random variables $\widehat{\Ys}^{H}_n$ converge to $\Ys_H$ only in distribution and one can not conclude that $\Ys_H$ is $\Ts_\infty$-measurable. Anyhow, this naive idea leads to some interesting results.

\begin{lemma}\label{01l}
Consider a sequence of positive numbers $s_n$. For any subsequence $n_k$, the sets $\{|\widehat{\Ys}^{H}_{n_k}|>s_{n_k}\, \textrm{i.o.}\}$ and $\{|\widehat{\Ys}^{H}_{n_k}|>s_{n_k}\, \textrm{ult.}\}$ have probability $0$ or $1$.
\end{lemma}
\begin{proof}
By the Kolmogorov 0-1 law, it is enough to prove that both sets belong to the tail $\sigma$-field. More precisely, we have to prove that for all $m\geq 1$, both sets belong to $\sigma(\xi_i,i>m)$. 

First, fix $m\geq 1$. Since $i_k\rightarrow\infty$ and $n_k\rightarrow\infty$ when $k$ goes to infinity, we can find $k_m$ such that $i_{n_{k_m}}> m$. On the other hand, we have that:
$$\{|\widehat{\Ys}^{H}_{n_k}|>s_{n_k}\}\in \sigma\{\xi_i, i\geq i_{n_k}\}.$$
Thus,
$$\bigcup\limits_{k\geq \ell}\{|\widehat{\Ys}^{H}_{n_k}|>s_{n_k}\}\in\sigma\{\xi_i, i\geq i_{n_\ell}\}\quad\textrm{and}\quad\bigcap\limits_{k\geq \ell}\{|\widehat{\Ys}^{H}_{n_k}|>s_{n_k}\}\in\sigma\{\xi_i, i\geq i_{n_\ell}\}.$$
We conclude that:
$$\{|\widehat{\Ys}^{H}_{n_k}|>s_{n_k}\, \textrm{i.o.}\}=\bigcap\limits_{\ell\geq k_m}\bigcup\limits_{k\geq \ell}\{|\widehat{\Ys}^{H}_{n_k}|>s_{n_k}\}\in\sigma\{\xi_i, i\geq i_{n_{k_m}}\}\subseteq\sigma\{\xi_i, i>m\}$$
and
$$\{|\widehat{\Ys}^{H}_{n_k}|>s_{n_k}\, \textrm{ult.}\}=\bigcup\limits_{\ell\geq k_m}\bigcap\limits_{k\geq \ell}\{|\widehat{\Ys}^{H}_{n_k}|>s_{n_k}\}\in\sigma\{\xi_i, i\geq i_{n_{k_m}}\}\subseteq\sigma\{\xi_i, i>m\}$$
and the proof is completed.
\end{proof}
We denote $H_c:=2h_c-1/2\in(1/2,1)$, where $h_c$ is chosen like in Corollary \ref{hc}.
\begin{lemma}\label{01l2}
Let $n_k$ be a strictly increasing sequence of positive integers. For any $H<H_c$, there exists $\ep>0$:
$$P(|\widehat{\Ys}^{H}_{n_k}|>g_{n_k}^H(1-\ep)\, \textrm{ult.})=0$$
and for any $H>H_c$, there exists $\delta>0$:
$$P(|\widehat{\Ys}^{H}_{n_k}|>g_{n_k}^H(1+\delta)\, \textrm{i.o.})=1.$$
\end{lemma}
\begin{proof}
Fix $H<H_c$. By Corollary \ref{hc}, we can choose $\ep$ such that $$\sum\limits_{k=1}^\infty\rho_h^2(k)<\frac{(1-\ep)^2}4<\frac14.$$ 
Using Fatou's lemma and Tchebysheff's inequality, we obtain:
$$P(|\widehat{\Ys}^{H}_{n_k}|>g_{n_k}^H(1-\ep)\, \textrm{ult.})\leq \liminf\limits_{k\rightarrow\infty}P(|\widehat{\Ys}^{H}_{n_k}|>g_{n_k}^H(1-\ep))\leq \frac4{(1-\ep)^2}\sum\limits_{k=1}^\infty\rho_h^2(k)<1,$$
Thus, the first result is a consequence of Lemma \ref{01l}.

Now, fix $H>H_c$ and choose $\delta>0$ such that
$$\sum\limits_{k=1}^\infty\rho_h^2(k)>\frac{(1+\delta)^2}4>\frac14.$$ 
Another application of Fatou's lemma, implies that:
$$P(|\widehat{\Ys}^{H}_{n_k}|>g_{n_k}^H(1+\delta)\, \textrm{i.o.})\geq \limsup\limits_{n\rightarrow\infty}P(|\widehat{\Ys}^{H}_{n_k}|>g_{n_k}^H(1+\delta)).$$
Appliying Paley-Zigmund inequality and Khintchine's inequality, we get:
\begin{align*}
\limsup\limits_{n\rightarrow\infty}P(|\widehat{\Ys}^{H}_{n_k}|>g_{n_k}^H(1+\delta))&=P(|\Ys_H|>g_H(1+\delta))\\
&\geq \frac{1}{3}\left(1-\left(\frac{(1+\delta)^2}{4\sum\limits_{k=1}^\infty\rho_h^2(k)}\right)\right)^2 >0,
\end{align*}
and the second statement follows from Lemma \ref{01l}.
\end{proof}

\begin{proof}[Proof of Theorem~\ref{asp}]
Since $\bar{\Ys}^H_n$ converges to $0$ in $L^2$, there is a subsequence $\bar{\Ys}_{n_k}^{H}$ convergent almost surely to $0$.
Let $H>H_c$ and $\delta>0$ as in Lemma \ref{01l2}. Note that
 \begin{align*}
  \{|\widehat{\Ys}^{H}_{n_k}|>(1+\delta)g_{n_k}^H\}=&\{|\widehat{\Ys}^{H}_{n_k}|>(1+\delta)g_{n_k}^H\}\cap\{|\bar{\Ys}_{n_k}^{H}|\leq\delta g_{n_k}^H\}\\
  &\cup\{|\widehat{\Ys}^{H}_{n_k}|>(1+\delta)g_{n_k}^H\}\cap\{|\bar{\Ys}_{n_k}^{H}|>\delta g_{n_k}^H\}\\
  \subseteq& \{|\Ys^H_{n_k}|>g_{n_k}^H\}\cup\{|\bar{\Ys}_{n_k}^{H}|>\delta g_{n_k}^H\},
 \end{align*}
which implies that
$$\{|\widehat{\Ys}^{H}_{n_k}|>(1+\delta)g_{n_k}^H\textrm{ i.o.}\}\subseteq\{|\Ys^H_{n_k}|>g_{n_k}^H\textrm{ i.o.}\}\cup\{|\bar{\Ys}_{n_k}^{H}|>\delta g_{n_k}^H\textrm{ i.o.}\}.$$
Since $\bar{\Ys}_{n_k}^{H}\xrightarrow{\textrm{a.s.}}0$ and $g_{n_k}^H\rightarrow g_H>0$, it follows that $P(\{|\bar{\Ys}_{n_k}^{H}|>\delta g_{n_k}^H\textrm{ i.o.}\})=0$. The desired statement is a consequence of Lemma \ref{01l2} and of the fact that 
$$\{|\Ys^H_{n_k}|>g_{n_k}^H\textrm{ i.o.}\}\subseteq \{|\Ys^H_{n}|>g_{n}\textrm{ i.o.}\}.$$
For $H<H_c$, let $\ep>0$ as in Lemma \ref{01l2}. Note that
 \begin{align*}
  \{|\Ys^H_{n_k}|>g_{n_k}^H\}=&\{|\Ys^H_{n_k}|>g_{n_k}^H\}\cap\{|\bar{\Ys}_{n_k}^{H}|\leq\ep g_{n_k}^H\}\\
  &\cup\{|\Ys^H_{n_k}|>g_{n_k}^H\}\cap\{|\bar{\Ys}_{n_k}^{H}|>\ep g_{n_k}^H\}\\
  \subseteq& \{|\widehat{\Ys}^{H}_{n_k}|>(1-\ep)g_{n_k}^H\}\cup\{|\bar{\Ys}_{n_k}^{H}|>\ep g_{n_k}^H\},
 \end{align*}
which implies that
$$\{|\Ys^H_{n_k}|>g_{n_k}^H\textrm{ ult.}\}\subseteq\{|\widehat{\Ys}^{H}_{n_k}|>(1-\ep)g_{n_k}^H\textrm{ ult.}\}\cup\{|\bar{\Ys}_{n_k}^{H}|>\ep g_{n_k}^H\textrm{ i.o.}\}.$$
Following the same reasoning as before, the desired statement is obtained using Lemma \ref{01l2} and the fact that 
$$\{|\Ys^H_{n_k}|>g_{n_k}^H\textrm{ ult.}\}\supseteq \{|\Ys^H_{n}|>g_{n}\textrm{ ult.}\}.$$
For the final statement, note that
 $$\{|\Ys^H_n|>g_n^H\textrm{ i.o.}\}\subseteq\{\exists\ n:\ |\Ys^H_n|>g_n^H\}$$
 and
 $$\{\exists\ n:\ |\Ys^H_n|>g_n^H\}\subseteq\{\exists\ n\in\{1,\ldots,N\}:\ |\Ys^H_n+a_n^{(N)}N^H|>g_n^H\textrm{ ult.}\}.$$
 Using \eqref{ap}, Fatou's lemma and Theorem \ref{asp} we obtain 
 $$\liminf_{N\to\infty}\frac{\As_\Ps^{(N,H)}}{2^{N-1}}\geq1$$
 and the result follows.
\end{proof}

% \begin{corollary}
%  If $H>2h_c-1/2$ then 
%  $$\lim_{N\to\infty}\frac{\As_\Ps^{(N,H)}}{2^{N-1}}=1.$$
% \end{corollary}
% 
% \begin{proof}
%  Note that
%  $$\{|\Ys^H_n|>g_n^H\textrm{ i.o.}\}\subseteq\{\exists\ n:\ |\Ys^H_n|>g_n^H\}$$
%  and
%  $$\{\exists\ n:\ |\Ys^H_n|>g_n^H\}\subseteq\{\exists\ n\in\{1,\ldots,N\}:\ |\Ys^H_n+a_n^{(N)}N^H|>g_n^H\textrm{ ult.}\}.$$
%  Using \eqref{ap}, Fatou's lemma and Theorem \ref{asp} we obtain 
%  $$\liminf_{N\to\infty}\frac{\As_\Ps^{(N,H)}}{2^{N-1}}\geq1$$
%  and the result follows.
% \end{proof}
\appendix
\section{Some technical results}
Consider the function $g_n:(0,n-1)\rightarrow (0,\infty)$ defined by $g_n(x)=x^{\frac12-H}\phi_n^H(x)$. The relation between this function and the coefficients $j_n^H(i)$ is given by \eqref{ej1}. The next lemma is the key ingredient when splitting the random variables $\Ys_n^H$. 
\begin{lemma}\label{mgn}
For each $n>1$, there exists a unique $x_n\in (0,n-1)$ such that the function $g_n$ is strictly decreasing in the interval $(0, x_n)$ and strictly increasing in the interval $(x_n,n-1)$. In addition, we have:
$$\lim\limits_{n\rightarrow\infty}\frac{x_n}{n-1}=H-\frac12.$$
\end{lemma}
\begin{proof}
Let $f_n:(1/(n-1),\infty)\rightarrow(0,\infty)$ be given by $f_n(x):=g_n(1/x)$. Note that
$$f_n'(y)=(H-\frac12)\, n\, (n y-1)^{H-\frac32}\left[1-\frac{(n-1)}{n}\left(1+\frac{1}{n-1-\frac{1}{y}}\right)^{\frac32-H}\right].$$
The function in the square parenthesis is strictly increasing and equal to $0$ in only one point, given by 
$$y_n=\left(n-1 -\frac{1}{\left(1+\frac{1}{n-1}\right)^{\frac{2}{3-2H}}-1}\right)^{-1}>\frac{1}{n-1}.$$
We deduce that $f_n$ is strictly decreasing in $(1/(n-1),y_n)$ and strictly increasing in $(y_n,\infty)$. Setting $x_n:=1/y_n$, the desired monotonicity properties  of $g_n$ follow. The last result is obtained via a direct calculation of the limit of $x_n/(n-1)$.
\end{proof}
An important consequence of this lemma is given in the following corollary.
\begin{corollary}\label{iin}
Denote $i_n=\lfloor x_n \rfloor +1$. The function $I_n:\{1,...,n-1\}\rightarrow (0,\infty)$ given in Lemma \ref{ejg} is decreasing on $\{1,..,i_n-1\}$ and increasing on $\{i_n+1,...,n-1\}$.
\end{corollary}
\begin{proof}
Directly from the definitions and Lemma \ref{mgn}.
\end{proof} 
%%%%%%%%%%%%%%%%%%%%%%%%%%%%%%%%%%%%%%%%%%%%%%%%%%%%%%%%%%%%%%%%%%%%
The remaining results concern properties of the function $\rho_h$.
\begin{lemma}\label{mono}
For $k\geq 1$, the function $F_k:(\frac{1}{2},\frac{3}{4})\rightarrow \Rb_+$ given by $F_k(h):=\rho_h(k)$ is increasing. As a consequence:
$\lim_{h\rightarrow\frac{1}{2}+}\sum_{k=1}^\infty\rho_h^2(k)=0.$
\end{lemma}
\begin{proof}
For $k=1$, we have $2 F_1(h)=2^{2h}-2$ and the result follows. For $k>1$, we note that:
$$2F_k(h)=k^{2h}G_{\frac{1}{k}}(2h),$$
where for $\varepsilon\in(0,1)$ and $x>1$, $G_\varepsilon(x):=(1+\varepsilon)^x+(1-\varepsilon)^x-2$.
We see that $G_\varepsilon^{'}(x)=\ln(1+\varepsilon)e^{x\ln(1+\varepsilon)}+\ln(1-\varepsilon)e^{x\ln(1-\varepsilon)}$
and that $G_\varepsilon^{'}(x_*(\varepsilon))=0$ if and only if: 
$$x_*(\varepsilon)=\frac{\ln\left(\frac{\ln(1/(1-\varepsilon))}{\ln(1+\varepsilon)}\right)}{\ln\left(\frac{1+\varepsilon}{1-\varepsilon}\right)}.$$
Moreover, since the function $g$ given by $g(x):=(1+x)\ln(1+x)+(1-x)\ln(1-x)$ is increasing in $[0,1)$, we deduce that $G'_{\ep}(1)=g(\ep)>g(0)=0$. The latter also implies that
$$\frac{\ln(1/(1-\varepsilon))}{\ln(1+\varepsilon)}<\frac{1+\varepsilon}{1-\varepsilon}.$$
Taking logarithm on both sides of the inequality, we conclude that $x_*(\varepsilon)<1$. In this way, we showed that $G_\varepsilon$ is increasing in $(1,\infty)$. Then the first assertion follows. The second one is obtained applying the monotone convergence theorem. 
\end{proof}

\begin{lemma}\label{lbr}
For any $h\in(\frac{1}{2},\frac{3}{4})$: 
$$\rho_h(k)\geq \frac{h (2h-1)}{2\,k^{2-2h}},\quad k\geq 1.$$
In particular:
$$\sum\limits_{k=1}^\infty\rho_h^2(k)\geq \frac{h^2 (2h-1)^2}{4}\, \zeta(4-4h),$$
where $\zeta$ is the Riemann zeta function. Consequently, $\lim_{h\rightarrow \frac{3}{4}-}\sum_{k=1}^\infty\rho_h^2(k)=\infty.$
\end{lemma}
\begin{proof}
Let's first prove that the function $f_h:[0,1)\rightarrow\Rb$ given by:
$$f_h(x):=(1+x)^{2h}+(1-x)^{2h}-2- h(2h-1)x^2$$
is positive. Note that $f_h^{'}(x)=2h \left((1+x)^{2h-1}-(1-x)^{2h-1}-(2h-1)x\right)$
and $f_h^{''}(x)=2h(2h-1) \left((1+x)^{2h-2}+(1-x)^{2h-2}-1\right)\geq 0.$
Since $f_h^{'}$ is increasing and $f_h^{'}(0)=0$, we conclude that $f_h^{'}$ is positive. Thus, $f_h$ is increasing and since $f_h(0)=0$, the claim is proved.
The result follows since:
$$\rho_h(k)= \frac{k^{2h}f_h\left(\frac{1}{k}\right)}{2}\, +\frac{h (2h-1)}{2\,k^{2-2h}}.$$
\end{proof}
As a consequence of the two previous results, we obtain the following corollary. 
\begin{corollary}\label{hc}
There exists $\frac{1}{2}<h_c<\frac{3}{4}$ such that:
\begin{enumerate}
 \item For all $h\in(\frac{1}{2},h_c)$: $\sum\limits_{k=1}^\infty\rho_h^2(k)<\frac14.$
 \item For all $h\in(h_c,\frac{3}{4})$: $\sum\limits_{k=1}^\infty\rho_h^2(k)>\frac14.$
\end{enumerate}
\end{corollary}
\begin{proof}
First, we observe that the continuity and the monotonicity of the autocovariance functions with respect to $h$ and the monotone convergence theorem imply the continuity of the function $h\mapsto\sum_{k=1}^\infty\rho_h^2(k)$. The statements in the corollary follow from the mean value theorem and Lemmas \ref{mono} and \ref{lbr}.
\end{proof}

\bibliographystyle{acm}
\bibliography{reference}
\end{document}